\documentclass[12pt]{amsart}  

\usepackage{custom_preamble} 

\begin{document}

\title{Structure in sets with logarithmic doubling}

\author{T. Sanders}
\address{Department of Pure Mathematics and Mathematical Statistics\\
University of Cambridge\\
Wilberforce Road\\
Cambridge CB3 0WB\\
England } \email{t.sanders@dpmms.cam.ac.uk}

\begin{abstract}
Suppose that $G$ is an abelian group, $A \subset G$ is finite with $|A+A| \leq K|A|$ and $\eta \in (0,1]$ is a parameter.  Our main result is that there is a set $\mathcal{L}$ such that
\begin{equation*}
|A \cap \Span(\mathcal{L})| \geq K^{-O_\eta(1)}|A| \textrm{ and } |\mathcal{L}| = O(K^\eta\log |A|).
\end{equation*}
We include an application of this result to a generalisation of the Roth-Meshulam theorem due to Liu and Spencer.
\end{abstract}

\maketitle

\section{Introduction}

Suppose that $G$ is an abelian group.  We are interested in the structure of sets with small doubling, the prototypical examples of which are coset progressions.  A set $M$ is a \emph{$d$-dimensional coset progression} if it can be written in the form
\begin{equation*}
M=H+P_1+\dots+P_d
\end{equation*}
where $H \leq G$ and $P_1,\dots,P_d$ are arithmetic progressions.  It is easy to see that if $A$ is a proportion $\delta$ of a $d$-dimensional coset progression then $|A+A| \leq \delta^{-1}2^d|A|$ -- $A$ has `small doubling'.  Remarkably there is something of a converse to this.
\begin{theorem}[Green-Ruzsa-Fre{\u\i}man theorem]\label{thm.grf}
Suppose that $G$ is an abelian group and $A \subset G$ has $|A+A| \leq K|A|$. Then there is an $O_\varepsilon(K^{4+\varepsilon})$-dimensional coset progression $M$ such that $A \subset M$ and $|M| \leq \exp(O_\varepsilon(K^{4+\varepsilon}))|A|$.
\end{theorem}
This result is due to Green and Ruzsa \cite{greruz::0} building on Ruzsa's proof \cite{ruz::9} of Fre{\u\i}man's theorem \cite{fre::0} in the integers.  There are other proofs (see \cite{taovu::} for example) and a large body of literature which we shall not survey here.

Whilst this resolves the situation from a qualitative perspective, quantitatively things are far less well understood.  In \cite{shk::3} Shkredov noticed that one may hope to say something quantitatively stronger if one changes the structure sought to that of spans: recall that if $\mathcal{L} \subset G$ then 
\begin{equation*}
\Span(\mathcal{L}):=\{\sum_{x \in \mathcal{L}}{\sigma_x.x}: \sigma_x \in \{-1,0,1\} \textrm{ for all } x \in \mathcal{L}\}.
\end{equation*}
With this notation Shkredov established the following theorem.
\begin{theorem}
Suppose that $G$ is an abelian group and $A \subset G$ has $|A+A| \leq K|A|$.  Then there is a set $\mathcal{L}$ such that
\begin{equation*}
A \subset \Span(\mathcal{L}) \textrm{ and }|\mathcal{L}| = O(K \log |A|).
\end{equation*}
\end{theorem}
Of course a span is a type of coset progression and so once $K$ is about $\log^{1/3} |A|$ the above result supersedes  the Green-Ruzsa-Fre{\u\i}man theorem. 

As it stands the result is essentially best possible -- consider a set $A$ of $K$ generic points.  However,  if one weakens the containment hypothesis to merely correlation then one can hope to do better and to this end we shall prove the following.
\begin{theorem}\label{thm.subexpshk}
Suppose that $G$ is an abelian group, $A \subset G$ has $|A+A| \leq K|A|$ and $\eta \in (0,1]$.  Then there is a set $\mathcal{L}$ such that
\begin{equation*}
|A \cap \Span(\mathcal{L})| \geq K^{-O(\exp (O(\eta^{-1})))}|A| \textrm{ and }|\mathcal{L}| = O(K^{\eta} \log |A|).
\end{equation*}
\end{theorem}
The reader may wish to compare this with the (much stronger) polynomial Fre{\u\i}man-Ruzsa conjecture.

To illustrate the utility of Theorem \ref{thm.subexpshk} we address a natural generalisation of the Roth-Meshulam theorem \cite{mes::} considered by Liu and Spencer in \cite{liuspe::}. 
\begin{theorem}\label{thm.cvs}
Suppose that $\F$ is a finite field, $G:=\F^n$, $c_1,\dots,c_r \in \F^*$ are such that $c_1+\dots+c_r=0$, and $A \subset G$ contains no solutions to $c_1.x_1+\dots+ c_r.x_r=0$ with $x_1,\dots,x_r \in A$ pair-wise distinct.  Then
\begin{equation*}
|A|=O_{|\F|,r}(|\F|^n/n^{r-2}).
\end{equation*}
\end{theorem}
The requirement that the elements be pair-wise distinct rules out degenerate solutions introduced by having shorter sub-sums of the $c_i$s equal to zero.  Nevertheless, it should be noted that for a number of special equations better bounds are available.  For example, if $c_i=-c_{r-i}$ and $r$ is even then a simple application of the Cauchy-Schwarz inequality will give a power shaped saving in the bound on $|A|$.  The different `types' of equation are given a comprehensive analysis by Ruzsa in \cite{ruz::8} -- we shall not address ourselves to this problem here.

The result above is a special case of the work of Liu and Spencer from \cite{liuspe::} who considered $r$-fold sums in arbitrary abelian groups and (along with Zhao) generalised it further to systems of equations of complexity $1$ in \cite{liuspezha::}. 

Improving the bound in Theorem \ref{thm.cvs} in the case $r=3$ (and $|\F|=3$) is a well known open problem sometimes called the capset problem, as discussed in \cite{gre::9,crolev::,tao::5}.  We shall use Theorem \ref{thm.subexpshk} to show that there is a non-negative valued function $E(r)$ with
\begin{equation*}
E(r)=\Omega(\log r) \textrm{ for all } r \textrm{ sufficiently large,}
\end{equation*}
such that the following theorem holds.
\begin{theorem}\label{thm.mainapp}
Suppose that $\F$ is a finite field, $G:=\F^n$, $c_1,\dots,c_r \in \F^*$ are such that $c_1+\dots+c_r=0$, and $A \subset G$ contains no solutions to $c_1.x_1+\dots+ c_r.x_r=0$ with $x_1,\dots,x_r \in A$ pair-wise distinct.  Then
\begin{equation*}
|A|=O_{|\F|,r}(|\F|^n/n^{r-2+E(r)}).
\end{equation*}
\end{theorem}
We emphasise that $E(r)$ only becomes non-zero once $r$ is sufficiently large; with some care this can be taken to be $2^{20}$.

The paper now splits as follows. In the next section, \S\ref{sec.ft}, we record the basics of the Fourier transform and structure of the spectrum which we require for the proof of Theorem \ref{thm.subexpshk}. In \S\ref{sec.asyms} we prove an asymmetric version of Shkredov's theorem, and then in \S\ref{sec.css} a symmetry set version of Chang's theorem.  These results are combined with a proposition from \cite{san::03} to prove Theorem \ref{thm.subexpshk} in \S\ref{sec.main}.  Following this we record some results from additive combinatorics which we require for our application in \S\ref{sec.addcomb}. Theorem \ref{thm.mainapp} is then established in \S\ref{sec.appln}.

It should be remarked that around the same time as this paper was written Schoen in \cite{sch::1} came up with a far better way of using the same ingredients to prove the first good bounds for a Fre{\u\i}man-Ruzsa-type theorem, and then a little later in an additional unpublished argument\footnote{Personal communication.} was able to improve Theorem \ref{thm.mainapp}.

\section{The {F}ourier transform and the large spectrum}\label{sec.ft}

A good introduction to the Fourier transform may be found in Rudin \cite{rud::1}, and for our work the more modern reference \cite{taovu::} of Tao and Vu.  Suppose that $G$ is a locally compact abelian group endowed with a Haar measure $\mu_G$. We define the convolution of two functions $f,g \in L^1(\mu_G)$ point-wise by
\begin{equation*}
f \ast g(x):=\int{f(y)g(-y+x)d\mu_G(y)},
\end{equation*}
and write $\wh{G}$ for the dual group, that is the locally compact abelian group of homomorphisms from $G$ to $S^1:=\{z \in \C: |z|=1\}$.  Convolution operators are diagonalized by the Fourier transform: we define the Fourier transform of a function $f \in L^1(\mu_G)$ by
\begin{equation*}
\wh{f}:\wh{G} \rightarrow \C; \gamma \mapsto \int{f(x)\overline{\gamma(x)}d\mu_G(x)}.
\end{equation*}
If we declare $G$ as discrete then we take $\mu_G$ to be counting measure (that is the measure assigning mass $1$ to every element of $G$) and if we declare $G$ as compact then we take $\mu_G$ to be $\P_G$ the unique Haar probability measure on $G$.  When $G$ is finite it will be clear from context which measure we take.

Suppose now that $G$ is compact and $f \in L^1(G)$. The Hausdorff-Young inequality tells us that $|\wh{f}(\gamma)| \leq \|f\|_{L^1(G)}$ and so it is natural to define the \emph{$\delta$-large spectrum of $f$} to be
\begin{equation*}
\Spec_\delta(f):=\{\gamma \in \wh{G}:|\wh{f}(\gamma)|\geq \delta \|f\|_{L^1(G)}\}.
\end{equation*}
Chang initiated work studying the structure of the spectrum in \cite{cha::0} and this has since been refined by Shkredov in \cite{shk::4}.
\begin{proposition}[Chang's theorem]
Suppose that $G$ is a compact abelian group, $f \in L^1(G)$ and $\delta \in (0,1]$ is a parameter.  Then there is a set $\mathcal{L}$ such that
\begin{equation*}
\Spec_\delta(f) \subset \Span(\mathcal{L}) \textrm{ and }|\mathcal{L}| =O(\delta^{-2}\log \|f\|_{L^2(G)}^2\|f\|_{L^1(G)}^{-2}).
\end{equation*}
\end{proposition}
The functional version of this result can be read out of the proof in Chang's original paper but was popularised by Green.  

\section{An asymmetric version of Shkredov's theorem}\label{sec.asyms}

In this section we use Chang's theorem to prove the following asymmetric version of Shkredov's theorem.  The key idea is the introduction of a certain auxiliary function, which is a trick used in \cite[Theorem 6.10]{lopros::} for proving a result on very similar lines.
\begin{proposition}\label{prop.shkasym}
Suppose that $G$ is a discrete abelian group and $A \subset G$ is a finite non-empty
set with $|B+A| \leq K|A|$. Then there is a set $\mathcal{L}$ such that
\begin{equation*}
B \subset \Span(\mathcal{L}) \textrm{ and }|\mathcal{L}| = O(K \log |A|).
\end{equation*}
\end{proposition}
\begin{proof}
Throughout this proof the Fourier transform is the Fourier transform on the compact group $\wh{G}$.  

Define $h$ and $k$ by inversion so that $\wh{h}=1_{B+A}$ and $\wh{k}=1_{-A}$, and put $g:=hk$.  If $x \in B$ then $1_{B+A} \ast 1_{-A}(x) =|A|$, so 
\begin{equation*}
B \subset \{x: 1_{B+A} \ast 1_{-A}(x) \geq |A|\} = \Spec_{|A|/\|g\|_{L^1(\wh{G})}}(g).
\end{equation*}
Applying Chang's theorem to this we get a set $\mathcal{L}$ such that
\begin{equation*}
B \subset \Span(\mathcal{L}) \textrm{ and } |\mathcal{L}| = O(\|g\|_{L^1(\wh{G})}^2|A|^{-2}\log \|g\|_{L^2(\wh{G})}^2\|g\|_{L^1(\wh{G})}^{-2}).
\end{equation*}
This is an increasing function of $\|g\|_{L^1(\wh{G})}$ and $\|g\|_{L^2(\wh{G})}$ so it remains to provide upper bounds for these quantities.  First,
\begin{eqnarray*}
\|g\|_{L^1(\wh{G})} & = & \int{|hk|d\P_{\wh{G}}}\\ & \leq &\|h\|_{L^2(\wh{G})}\|k\|_{L^2(\wh{G})}\\ & = & \sqrt{|B+A|.|-A|} \leq \sqrt{K}|A|,
\end{eqnarray*}
by the Cauchy-Schwarz inequality and Parseval's theorem.  Secondly
\begin{eqnarray*}
\|g\|_{L^2(\wh{G})}^2 & = &\|1_{B+A} \ast 1_{-A}\|_{\ell^2(G)}^2\\ & \leq &\|1_{B+A} \ast 1_{-A}\|_{\ell^\infty(G)}\|1_{B+A} \ast 1_{-A}\|_{\ell^1(G)}\\ & = & |B+A||-A|^2\leq K|A|^3
\end{eqnarray*}
by Parseval's theorem and then H{\"o}lder's inequality.  It follows that
\begin{equation*}
|\mathcal{L}| = O((\sqrt{K}|A|)^2|A|^{-2}\log (K|A|^3/(\sqrt{K}|A|)^2))=O(K\log |A|)
\end{equation*}
as required.
\end{proof}

\section{Structure in symmetry sets}\label{sec.css}

Recall from \cite{taovu::} that if $G$ is a discrete abelian group, $A \subset G$ is a finite non-empty set and $\eta \in (0,1]$ then the \emph{symmetry set of $A$ at threshold $\eta$} is
\begin{equation*}
\Sym_\eta(A):=\{x \in G: 1_A \ast 1_{-A}(x) \geq \eta |A|\}.
\end{equation*}
Symmetry sets are essentially dual to spectra so it should come as no surprise that they also have a structure theorem along the lines of Chang's theorem.
\begin{proposition}[Chang's theorem for symmetry sets]\label{prop.symchang}
Suppose that $G$ is a discrete abelian group, $A \subset G$ is a finite non-empty set and $\eta \in (0,1]$ is a parameter. Then there is a set $\mathcal{L}$ such that
\begin{equation*}
\Sym_\eta(A) \subset \Span(\mathcal{L}) \textrm{ and } |\mathcal{L}| = O(\eta^{-2}\log |A|).
\end{equation*}
\end{proposition}
\begin{proof}
Symmetry sets are dual to spectra in the sense that $\Sym_\eta(A) = \Spec_\eta(f)$ where $f:=|\wh{1_A}|^2$.  To see this note that
\begin{equation*}
\|f\|_{L^1(\wh{G})}=\||\wh{1_A}|^2\|_{L^1(\wh{G})} = \|\wh{1_A}\|_{L^2(\wh{G})}^2 = \|1_A\|_{\ell^2(G)}^2 = |A|
\end{equation*}
by Parseval's theorem.  In light of this we apply Chang's theorem to get that $\Sym_{\eta}(A)$ is contained in $ \Span(\mathcal{L})$ for some set $\mathcal{L}$ with
\begin{equation*}
 |\mathcal{L}|=O(\eta^{-2}\log \|f\|_{L^2(\P_G)}^2\|f\|_{L^1(\P_G)}^{-2}).
\end{equation*}
The argument of the logarithm may then be bounded above by H{\"o}lder's inequality and the Hausdorff-Young inequality:
\begin{equation*}
 \|f\|_{L^2(\P_G)}^2\|f\|_{L^1(\P_G)}^{-2}\leq \|f\|_{L^\infty(\wh{G})}\|f\|_{L^1(\wh{G})}^{-1} =\||\wh{1_A}|^2\|_{L^\infty(\wh{G})}/ |A|\leq |A|.
\end{equation*}
The result is proved.
\end{proof}

\section{The proof of Theorem \ref{thm.subexpshk}}\label{sec.main}

In light of Proposition \ref{prop.symchang} we should like to show that if $A$ has small doubling then it correlates with a symmetry set having large threshold.  To this end we recall the following result.
\begin{proposition}[{\cite[Proposition 1.3]{san::03}}]\label{prop.core}
Suppose that $G$ is a discrete abelian group, $A$ is a non-empty subset of $G$ with $|A+A| \leq K|A|$, and $\epsilon \in (0,1]$ is a parameter. Then there is a non-empty set $A'\subset A$ such that
\begin{equation*}
|\Sym_{1-\epsilon}(A'+A)| \geq \exp(-K^{O(1/\log (1/(1-\epsilon)))}\log K)|A|.
\end{equation*}
\end{proposition}
In fact the above is true for non-abelian groups as well (with the obvious changes of sums to products) but our other results are not.   We shall use it in the range when $\epsilon$ is close to $1$; the fact that it still has content in this region is an idea due to Tao.

We now have all the ingredients necessary for the proof of our main result.
\begin{proof}[Proof of Theorem \ref{thm.subexpshk}]
We begin by applying Proposition \ref{prop.core} with parameter $\epsilon = 1-K^{\eta/2}$ to get that there is a non-empty set $A'\subset A$ with
\begin{equation*}
 |\Sym_{K^{-\eta/2}}(A'+A)| \geq  K^{-\exp(O(\eta^{-1}))}|A|.
\end{equation*}
We apply Proposition \ref{prop.shkasym} to get a set $\mathcal{L}$ such that
\begin{equation*}
S:=\Sym_{K^{-\eta/2}}(A'+A) \subset \Span(\mathcal{L}) \textrm{ and } |\mathcal{L}| = O(K^{\eta}\log |A|).
\end{equation*}
On the other hand
\begin{equation*}
1_{A'+2A} \ast 1_{-A}(x) \geq |A|.1_{A'+A}(x) \textrm{ for all }x \in G,
\end{equation*}
whence
\begin{eqnarray*}
|A|^2.K^{-\eta/2}|A'+A||S| & \leq & |A|^2.\langle 1_{A'+A} \ast 1_{-(A'+A)},1_S \rangle\\ & \leq & \langle 1_{A'+2A} \ast 1_{-A} \ast 1_A \ast 1_{-(A'+2A)},1_S \rangle\\ & \leq & \|1_{A' +2A} \ast 1_{-(A'+2A)} \ast 1_A\|_{\ell^1(G)}\|1_A \ast 1_S\|_{\ell^\infty(G)}\\ & = & |A'+2A|^2|A|\|1_A \ast 1_S\|_{\ell^\infty(G)}.
\end{eqnarray*}
Since $A' \subset A$ and $|A+A| \leq K|A|$ we have, by Pl{\"u}nnecke's inequality, that $|A'+2A| \leq K^3|A|$ and so
\begin{equation*}
\|1_A \ast 1_S\|_{\ell^\infty(G)} \geq K^{3-\eta/2}|S| \geq K^{-\exp(O(\eta^{-1}))}|A|.
\end{equation*}
It follows that there is some $x$ such that $|A\cap (x+S)| \geq K^{-\exp(O(\eta^{-1}))}|A|$, but then $x+S \subset \Span(\mathcal{L}')$ where $\mathcal{L}':=\mathcal{L}\cup \{x\}$.  The result is proved.
\end{proof}

\section{Some tools of the trade in additive combinatorics}\label{sec.addcomb}

In this section we shall record some of the standard tools used in additive combinatorics for the purposes of proving Theorem \ref{thm.subexpshk} in the next section.

Chang's theorem from \S\ref{sec.ft} is proved using Rudin's inequality and in our context this may be seen as an estimate for the higher energy norms of the spectrum.  Shkredov in \cite{shk::4} encoded this idea formally and we shall now record a weak version of one of his results saying that the large spectrum has large additive energy; we include a proof since it is so short.
\begin{proposition}\label{prop.shk}
Suppose that $G$ is a compact abelian group, $A \subset G$ has density $\alpha>0$ and $S \subset \Spec_\delta(A)$.  Then
\begin{equation*}
E(S):=\|1_S \ast 1_{-S}\|_{\ell^2(\wh{G})}^2 \geq \delta^8\alpha |S|^4.
\end{equation*}
\end{proposition}
\begin{proof}
We begin by applying Plancherel's theorem and H{\"o}lder's inequality to the inner product
\begin{equation*}
|\langle\wh{1_A}1_S,\wh{1_A}\rangle_{\ell^2(\wh{G})}| = | \langle 1_A \ast \wh{1_S},1_A\rangle_{L^2(G)}| \leq \|1_A \ast \wh{1_S}\|_{L^4(G)}\|1_A\|_{L^{4/3}(G)}.
\end{equation*}
By a trivial instance of Young's inequality and Parseval's theorem we have
\begin{equation*}
\|1_A \ast \wh{1_S}\|_{L^4(G)} \leq \|1_A\|_{L^1(G)}\|\wh{1_S}\|_{L^4(G)}=\alpha E(S)^{1/4},
\end{equation*}
and even more trivially we have $\|1_A\|_{L^{4/3}(G)}  \leq \alpha^{3/4}$.  On the other hand
\begin{equation*}
\langle\wh{1_A}1_S,\wh{1_A}\rangle_{\ell^2(\wh{G})}\geq \delta^2\alpha^2|S|,
\end{equation*}
from which the result follows on rearranging.
\end{proof}
In \cite{shk::4} Shkredov extends the above in two ways: first, by considering different powers in H{\"o}lder's inequality he gets a lower bound on the $2k$-th energy (that is $\|\wh{1_S}\|_{L^{2k}(G)}^{2k}$); secondly, by dyadically decomposing the range of $|\wh{1_A}|$, he improves the $\delta^8$ to $\Omega(\delta^4)$.  

It is easy to see from Parseval's inequality that $S$ has size at most $\delta^{-2}\alpha^{-1}$; the reader should think of the situation when the size is close to this, $\delta$ is fixed but possibly small and $\alpha\rightarrow 0$.  Then $|S|$ tends to infinity in size and $E(S) \geq \delta^{O(1)}|S|^3$ -- it has large additive energy.

In the situation described above we have the celebrated Balog-Szemer{\'e}di-Gowers theorem (see \cite{balsze::} and \cite{gow::4}) which we now recall.
\begin{theorem}
Suppose that $G$ is an abelian group and $A \subset G$ has $E(A)\geq c|A|^3$.  Then there is a subset $A' \subset A$ such that
\begin{equation*}
|A'| = \Omega(c^{O(1)}|A|) \textrm{ and } |A'+A'| = O(c^{-O(1)}|A'|).
\end{equation*}
\end{theorem}
Gowers \cite{gow::4} made the important observation that this could then naturally be combined with a Fre{\u\i}man-type theorem in many applications and our present work is another such example.

Finally we need to record how we pass from large Fourier coefficients to increased density on a subspace when $G:=\F^n$.  The key to the simplicity of this in the finite field model is the following easy calculation.  Suppose that $W \leq \wh{G}$.  Then
\begin{equation*}
\wh{\P_{W^\perp}}(\gamma) = \begin{cases} 1 & \textrm{ if } \gamma \in W,\\ 0 & \textrm{ otherwise.}\end{cases}.
\end{equation*}
We are now in a position to record the `Roth-Meshulam' increment lemma.
\begin{lemma}[$\ell^\infty(\wh{G})$-increment lemma]\label{lem.structinc}
Suppose that $\F$ is a finite field, $G:=\F^n$, $A \subset G$ has density $\alpha$ and $\sup_{\gamma \neq 0_{\wh{G}}}{|\wh{1_A}(\gamma)|} \geq \epsilon \alpha$. Then there is a subspace $V \leq G$ wth $\cod V=1$ and
\begin{equation*}
\|1_ A \ast \P_V\|_{L^\infty(G)} \geq \alpha(1+\epsilon/2).
\end{equation*}
\end{lemma}
\begin{proof}
We do the obvious thing and define $V=\{\gamma\}^\perp$ so that
\begin{equation*}
((1_A - \alpha) \ast \P_V)^\wedge(\gamma) = \wh{1_A}(\gamma),
\end{equation*}
whence by the Hausdorff-Young inequality we have
\begin{equation*}
\|(1_A - \alpha) \ast \P_V\|_{L^1(G)} \geq \epsilon \alpha.
\end{equation*}
On the other hand
\begin{equation*}
\int{((1_A - \alpha) \ast \P_V)d\P_G} = 0,
\end{equation*}
whence
\begin{equation*}
2\sup_{x \in G}{(1_A - \alpha) \ast \P_V(x)} \geq \epsilon \alpha.
\end{equation*}
The result follows on dividing by $2$ and adding $\alpha$ to both sides.
\end{proof}
It is also possible to get a very large correlation with a subspace if one has a large $\ell^2(\wh{G})$ mass of $\wh{1_A}$.  This is an idea introduced by Szemer{\'e}di in \cite{sze::2} and encoded in the model setting by the following lemma.
\begin{lemma}[$\ell^2(\wh{G})$-increment lemma]\label{lem.structincl2}
Suppose that $\F$ is a finite field, $G:=\F^n$, $A \subset G$ has density $\alpha>0$ and $W \leq \wh{G}$ is such that $\sum_{\gamma \in W}{|\wh{1_A}(\gamma)|^2} \geq \epsilon \alpha$.  Then there is a subspace $V \leq G$ wth $\cod V=\dim W$ and
\begin{equation*}
\|1_ A \ast \P_V\|_{L^\infty(G)} \geq \epsilon.
\end{equation*}
\end{lemma}
\begin{proof}
We do the obvious thing and define $V=W^\perp$ and so
\begin{equation*}
(1_A \ast \P_V)^\wedge(\gamma) = \wh{1_A}(\gamma) \textrm{ whenever }\gamma \in W.
\end{equation*}
Thus by Parseval's theorem and the hypothesis we have that
\begin{equation*}
\int{(1_A \ast \P_V)^2d\P_G} =\sum_{\gamma \in W}{|\wh{1_A}(\gamma)|^2} \geq \epsilon\alpha.
\end{equation*}
The result then follows by H{\"o}lder's inequality and the fact that
\begin{equation*}
\int{1_A \ast \P_Vd\P_G} = \alpha,
\end{equation*}
on dividing by $\alpha$.
\end{proof}

\section{Proof of Theorem \ref{thm.cvs}}\label{sec.appln}

The argument follows the usual iterative method pioneered by Roth \cite{rot::0} and exposed as particularly elegant in $\F_3^n$ by Meshulam in \cite{mes::}. The key quantity of interest is the number of solutions to the given equation.

Suppose that $\F$ is a finite field, $G:=\F^n$, $c  \in (\F^*)^r$ and $A \subset G$.  Then we write
\begin{equation*}
\Lambda_c(A):=\int{1_{-c_1.A}(c_2.x_2+\dots+ c_r.x_r)\prod_{i=2}^r{1_A(x_i)}d\P_G(x_2)\dots d\P_G(x_r)}.
\end{equation*}
Using the inversion formula, we may put
\begin{equation*}
1_A(x_i)=\sum_{\gamma_i \in \wh{G}}{\wh{1_A}(\gamma_i)\gamma_i(x_i)} \textrm{ for all }x_i \in G.
\end{equation*}
We insert this expression for $1_A$ into each instance in $\Lambda_c(A)$, and via the orthogonality relations get that $c_i.\gamma_i=c_j.\gamma_j =:\gamma$ for all $i,j$.  This gives a Fourier expression for $\Lambda_c(A)$ as follows:
\begin{equation}\label{eqn.fourier}
\Lambda_c(A)=\sum_{\gamma \in \wh{G}}{\prod_{i=1}^r{\wh{1_A}(c_i^{-1}.\gamma)}}.
\end{equation}
Of course, we shall use the above Fourier expression in the following driving lemma for our argument.
\begin{lemma}[Iteration lemma]\label{lem.itbitty}
There is a non-negative valued function $\nu$ with $\nu(r)=\Omega(r^{-1}\log r)$ for $r$ greater than some absolute constant such that if  $\F$ is a finite field, $G:=\F^n$, $c_1,\dots,c_r \in \F^*$ and $A \subset G$ has density $\alpha>0$, then at least one of the following is true:
\begin{enumerate}
\item (Many solutions) we have the lower bound $\Lambda_c(A) \geq \alpha^r/2$;
\item (Small correlation with low co-dimension subspace) there is a subspace $V \leq G$ with $\cod V=1$ such that 
\begin{equation*}
\|1_A \ast \P_V\|_{L^\infty(G)}\geq \alpha(1+\Omega(\alpha^{(1-\nu(r))/(r-2)}));
\end{equation*}
\item (Large correlation with a large co-dimension subspace) there is a subspace $V \leq G$ with $\cod V = O_r(\alpha^{-1/2(r-2)})$ such that
\begin{equation*}
\|1_A \ast \P_V\|_{L^\infty(G)}\geq \Omega(\alpha^{1/2}).
\end{equation*}
\end{enumerate}
\end{lemma}
\begin{proof}
If we are in the first case of the lemma we are done; assume not so that from (\ref{eqn.fourier}) we get
\begin{equation*}
|\sum_{\gamma \in \wh{G}}{\prod_{i=1}^r{\wh{1_A}(c_i^{-1}.\gamma)}}| \leq \alpha^r/2.
\end{equation*}
As usual we extract the trivial mode: we have $\wh{1_A}(\gamma)=\alpha$ whence
\begin{equation*}
|\alpha^r+\sum_{\gamma \neq 0_{\wh{G}}}{\prod_{i=1}^r{\wh{1_A}(c_i^{-1}.\gamma)}}| \leq \alpha^r/2.
\end{equation*}
Thus, by the triangle inequality we get
\begin{equation*}
\sum_{\gamma \neq 0_{\wh{G}}}{\prod_{i=1}^r{|\wh{1_A}(c_i^{-1}.\gamma)|}} \geq \alpha^r/2.
\end{equation*}
We apply the $r$-function version of H{\"o}lder's inequality to this to get that
\begin{equation*}
\prod_{i=1}^r{\left(\sum_{\gamma \neq 0_{\wh{G}}}{|\wh{1_A}(c_i^{-1}.\gamma)|^r}\right)^{1/r}} \geq \alpha^2/2.
\end{equation*}
Now, each $c_i \in \F^*$ whence $c_i^{-1}.(\wh{G} \setminus \{0_{\wh{G}}\}) = (\wh{G} \setminus \{0_{\wh{G}}\})$ and
\begin{equation*}
\sum_{\gamma \neq 0_{\wh{G}}}{|\wh{1_A}(c_i^{-1}.\gamma)|^r}=\sum_{\gamma \neq 0_{\wh{G}}}{|\wh{1_A}(\gamma)|^r} \textrm{ for all } 1 \leq i \leq r.
\end{equation*}
Inserting this back into our inequality we see that each factor is the same and we get that
\begin{equation}\label{eqn.mass}
\sum_{\gamma \neq 0_{\wh{G}}}{|\wh{1_A}(\gamma)|^r} \geq \alpha^r/2.
\end{equation}
This inequality will let us analyse the large spectrum of $1_A$: write
\begin{equation*}
\epsilon:=\alpha^{1/(r-2)}/4 \textrm{ and } S:=\Spec_{\epsilon}(1_A)\setminus \{0_{\wh{G}}\}.
\end{equation*}
It follows from the definition of the spectrum and Parseval's theorem that 
\begin{eqnarray*}
\sum_{\gamma \not \in \Spec_{\epsilon}(1_A)}{|\wh{1_A}(\gamma)|^r} & \leq & (\epsilon \alpha)^{r-2}\sum_{\gamma \in \wh{G}}{|\wh{1_A}(\gamma)|^2}\\ & = & \alpha.4^{-(r-2)}.\alpha^{r-2}.\alpha \leq \alpha^r/4
\end{eqnarray*}
since $r \geq 3$.  Thus, by the triangle inequality and (\ref{eqn.mass}) we have
\begin{equation}\label{eqn.ass}
\sum_{\gamma \in S}{|\wh{1_A}(\gamma)|^r}\geq \alpha^r/2 - \sum_{\gamma \not \in \Spec_{\epsilon}(1_A)}{|\wh{1_A}(\gamma)|^r} \geq  \alpha^r/4.
\end{equation}
Now, suppose that $M \geq 1$ is a real to be optimised later.  If
\begin{equation*}
\sup_{\gamma \neq 0_{\wh{G}}}{|\wh{1_A}(\gamma)|} \geq \alpha^{-M/(r-2)r}\epsilon \alpha
\end{equation*}
then we shall be in the second case of the lemma by Lemma \ref{lem.structinc} once we optimise for $M$.  To proceed we therefore assume not so that
\begin{equation*}
\sup_{\gamma \neq 0_{\wh{G}}}{|\wh{1_A}(\gamma)|} \leq  \alpha^{-M/(r-2)r} \epsilon \alpha.
\end{equation*}
Inserting this into (\ref{eqn.ass}) we see that
\begin{equation*}
|S|.( \alpha^{-M/(r-2)r} \epsilon \alpha)^r \geq \alpha^r/4,
\end{equation*}
which can be rearranged to give
\begin{equation*}
|S| \geq \alpha^r.4^{-1}.\alpha^{-r}.4^r.\alpha^{-r/(r-2)}.\alpha^{M/(r-2)} =4^{r-1}\alpha^{(M-2)/(r-2)}.\alpha^{-1}.
\end{equation*}
Now, by Proposition \ref{prop.shk} $S$ has large additive energy.  Specifically
\begin{eqnarray*}
E(S) \geq \epsilon^8\alpha |S|^4& \geq& \alpha^{8/(r-2)}4^{-8}4^{r-1}\alpha^{(M-2)/(r-2)}|S|^3\\ &= & \alpha^{(M+6)/(r-2)}4^{r-9}|S|^3 = \Omega(\alpha^{O(M/r)}).
\end{eqnarray*}
It follows by the Balog-Szemer{\'e}di-Gowers theorem that there is some set $S' \subset S$ such that
\begin{equation*}
|S'| \geq \Omega(\alpha^{O(M/r)})|S| \textrm{ and } |S'+S'| \leq O(\alpha^{-O(M/r)})|S'|.
\end{equation*}
Now apply Theorem \ref{thm.subexpshk} with some parameter $\eta$ to get a set $\mathcal{L}$ such that
\begin{equation*}
|S' \cap \Span(\mathcal{L})| \geq \Omega(\alpha)^{O(\exp(O(\eta^{-1}))M/r)}|S'| \textrm{ and } |\mathcal{L}| = O(\alpha^{-O(\eta M/r)}\log |S'|).
\end{equation*}
This means that we may pick $\eta=\Omega(1/M)$ such that
\begin{equation*}
|S' \cap \Span(\mathcal{L})| \geq \alpha^{O(\exp(O(M))/r}|S'| \textrm{ and } |\mathcal{L}| = O(\alpha^{-1/4(r-2)}\log |S'|).
\end{equation*}
Write $W$ for the subspace generated by $\mathcal{L}$ and note that by the lower bound on $|S'|$ we thus have
\begin{equation*}
\sum_{\gamma \in W\setminus \{0_{\wh{G}}\}}{|\wh{1_A}(\gamma)|^2}\geq (\epsilon \alpha)^2|S' \cap  \Span(\mathcal{L})| = \Omega(  \alpha^{1+O(\exp(O(M)))/r}).
\end{equation*}
It follows that if $r \geq C$ for some absolute constant $C>0$ then we may pick $M=\Omega(\log r)$ in a way indendent of $A$ and $c$ such that
\begin{equation*}
\sum_{\gamma \in W\setminus \{0_{\wh{G}}\}}{|\wh{1_A}(\gamma)|^2}\geq \Omega(\alpha^{1+1/2}).
\end{equation*}
This is how the function $\nu$ is determined if $r \geq C$: $\nu(r)=M/r$.  On the other hand by Parseval's theorem we have that
\begin{equation*}
|S'| \leq |S| \leq (\epsilon \alpha)^{-2}.\alpha \leq O(\alpha^{-O(1)}),
\end{equation*}
whence
\begin{equation*}
\dim W= O(\alpha^{1/4(r-2)}\log |S'|)=O(\alpha^{1/4(r-2)}\log \alpha^{-1}).
\end{equation*}
We now apply Lemma \ref{lem.structincl2} to get the third conclusion of the lemma.  If $r \leq C$ then $\nu(r)=0$ and we simply note that $S$ is, in any case, non-empty and apply Lemma \ref{lem.structinc} to any character in this set to get the conclusion.
\end{proof}

With the above lemma we are ready to apply the usual iterative method.
\begin{proof}[Proof of Theorem \ref{thm.mainapp}]
We proceed by creating a sequence of subspaces $G=:V_0 \geq V_1 \geq \dots \geq V_k$ and sets $A_i \subset V_i$ with density $\alpha_i$ such that
\begin{equation}\label{eqn.itpos}
\Lambda_c(A) \geq |G:V_i|^{r-1}\Lambda_c(A_i) \textrm{ and } \alpha_i \geq \alpha_0.
\end{equation}
We begin by setting $A_0:=A$ and suppose that we have defined $A_i$ and $V_i$.  We apply Lemma \ref{lem.itbitty}.  If we are in the first or third cases we shall terminate.  If we are in the second case we have some $x \in V_i$ and $V_{i+1} \leq V_i$ of codimension $1$ such that
\begin{equation*}
\int{1_{x+A_i}d\P_{V_{i+1}}}\geq \alpha_i(1+\alpha_i^{(1-\nu(r))/(r-2)}).
\end{equation*}
We set $A_{i+1}:=(x+A_i)\cap V_{i+1}$.  Since $c_1+\dots+c_r=0$ we certainly have (\ref{eqn.itpos}). However, we also have that
\begin{equation*}
\alpha_{i+1} \geq \alpha_i(1+\Omega(\alpha_i^{(1-\nu(r))/(r-2)})).
\end{equation*}
It follows that after $I=O(\alpha_i^{-(1-\nu(r))/(r-2)}))$ iterations we have $\alpha_{i+I(i)} \geq 2\alpha_i$.  However, since the density is always at most $1$ the iteration must terminate within
\begin{equation*}
O(\alpha_0^{-(1-\nu(r))/(r-2)}))+O((2\alpha_0)^{-(1-\nu(r))/(r-2)}))+O((4\alpha_0)^{-(1-\nu(r))/(r-2)}))+\dots
\end{equation*}
steps.  Summing the geometric progression we see that we are either in the first or third cases of the lemma within
\begin{equation*}
O_r(\alpha^{-(1-\nu(r))/(r-2)})
\end{equation*}
iterations.  In the first case we see trivially that
\begin{eqnarray*}
\Lambda_c(A) \geq |G:V_i|^{r-1}\Lambda_c(A_i)& \geq& |G:V_i|^{r-1}\alpha_i^{r}/2\\ & \geq & \exp(-O_{|\F|,r}(\alpha^{-(1-\nu(r))/(r-2)})).
\end{eqnarray*}
On the other hand since $A$ contains no solutions to $c_1.x_1+\dots+c_r.x_r=0$ with $x_1,\dots,x_r \in A$ pair-wise distinct we see that
\begin{equation*}
\Lambda_c(A)=O_r(|G|^{-1})
\end{equation*}
and it follows that
\begin{equation}\label{eqn.out}
\alpha = O_{|\F|,r}(n^{(r-2)/(1-\nu(r))}).
\end{equation}
 Finally, if we terminate in the third case of the iteration lemma then we get a space $V \leq V_i$ such that
\begin{equation*}
|G:V| = |G:V_i|.|V_i:V| = O_{|\F|,r}(\alpha^{-(1-\nu(r))/(r-2)})
\end{equation*}
and the density of $A$ on $V$ is $\Omega(\alpha^{1/2})$.  If $\log |G:V| \geq \log |G|/2$ then it follows that we have the bound (\ref{eqn.out}) again; otherwise apply Theorem \ref{thm.cvs} to see that
\begin{equation*}
\alpha=O_{|\F|,r}(n^{2(r-2)}).
\end{equation*}
The result follows in view of the definition of $\nu$.
\end{proof}

\section*{Acknowledgements}

The author should like to thank Ben Green and Terry Tao for may useful conversations, Tomasz Schoen for bringing the paper \cite{sch::1} to the author's attention, Ilya Shkredov for remarks concerning Theorem \ref{thm.subexpshk}, and two anonymous referees for many useful remarks concerning generalisations and exposition.

\bibliographystyle{alpha}

\bibliography{references}

\end{document}